\documentclass[12pt]{amsart}
\usepackage{graphicx}
\usepackage[latin1]{inputenc}
\usepackage{amsmath,amssymb,amsfonts,mathrsfs}
\usepackage{amsthm,enumerate}

\usepackage{psfrag,color}
\usepackage[small,hang,bf,up,sf]{caption}

\newtheorem*{lemma}{Lemma A}
\newtheorem*{lemmaB}{Lemma B}
\newtheorem{tma}{Theorem}

\newtheorem{lma}[tma]{Lemma}
\newtheorem{prop}[tma]{Proposition}
\newtheorem{obs}[tma]{Remark}

\newcommand{\D}{\mathbb D}
\newcommand{\Z}{\mathbb Z}
\newcommand{\N}{\mathbb N}

\newcommand{\R}{\mathbb R}
\newcommand{\B}{\mathcal B}
\newcommand{\T}{\mathbb T}
\newcommand{\Di}{\mathcal D}
\newcommand{\A}{\mathcal A}
\newcommand{\F}{\mathcal F}

\newcommand{\G}{\mathcal G}
\newcommand{\V}{\mathcal V}

\newcommand{\Hi}{\mathsf{H^\infty}}
\newcommand{\Ima}{\operatorname{Im}}
\newcommand{\Rea}{\operatorname{Re}}
\newcommand{\Li}{\mathsf{L^\infty}}
\newcommand{\Lu}{\mathsf{L^1}}

\newcommand{\Bm}{\mathsf{BMO}}
\newcommand{\Mx}{\mathscr{M}}

\begin{document}

\title[Interpolating sequences for analytic selfmappings...]{Interpolating sequences for analytic selfmappings of the disc.}
\author[Pere Menal]{Pere Menal Ferrer}
\address{Pere Menal Ferrer, Departament de Matem\`{a}tiques, Universitat Aut\`{o}noma de Barcelona, 08193 Bellaterra, Barcelona, Spain}
\email{pmenal@mat.uab.cat}
\author[Nacho Monreal]{Nacho Monreal Galán}
\address{Nacho Monreal Galán, Departament de Matem\`{a}tiques, Universitat Aut\`{o}noma de Barcelona, 08193 Bellaterra, Barcelona, Spain}
\email{nachomg@mat.uab.es}
\author[Artur Nicolau]{Artur Nicolau}
\address{Artur Nicolau, Departament de Matem\`{a}tiques, Universitat Aut\`{o}noma de Barcelona, 08193 Bellaterra, Barcelona, Spain}
\email{artur@mat.uab.cat}

\thanks{
The first author is supported in part by grant 2005SGR00317. The
second and the third authors are supported in part by grants
MTM2005-00544, MTM2008-00145 and 2009SGR420.}

\begin{abstract}
Schwarz's Lemma leads to a natural interpolation problem for
analytic functions from the disc into itself. The corresponding
interpolating sequences are geometrically described in terms of a
certain hyperbolic density.
\end{abstract}

\maketitle

\section{Introduction.}
\noindent Let $\Hi$ be the algebra of bounded analytic functions
in the unit disc $\D$ of the complex plane.  Let
$$\B=\left\{f\in\Hi:\ \|f\|_\infty=\sup_{z\in\D}|f(z)|\leq1\right\}$$
be its closed unit ball. Given two sets of points
$\{z_1,\dots,z_N\}$ and $\{w_1,\dots,w_N\}$ in the unit disc, the
Nevanlinna-Pick interpolation problem consists in finding $f\in\B$
with $f(z_n)=w_n$, $n=1,\dots,N$. Nevanlinna and Pick
independently proved that the interpolation problem has a solution
if and only if the matrix
$$\left(\frac{1-w_i\overline w_j}{1-z_i\overline z_j}\right)_{i,j=1,\dots,N}$$
is positive semidefinite (\cite{Nevan},\cite{Pick}).  This is a
very nice result which is the root of a very active research area
(see for instance the book by J.\ Agler and J.\ McCarthy
\cite{AgMc} and the references there), with connections with other
topics. However, in some concrete situations, as the one we will
present, it is not easy to verify the matrix condition and one
needs to use more direct methods.

\medskip

\noindent Let $\beta(z,w)$ be the hyperbolic distance between two
points $z,w\in\D$.  A sequence of points $\{z_n\}$ in the unit
disc is called an interpolating sequence if for any bounded
sequence of values $\{w_n\}$ there exists a function $f\in\Hi$
with $f(z_n)=w_n$, $n=1,2,\dots$  A celebrated result of L.\
Carleson \cite{Carleson} asserts that $\{z_n\}$ is an
interpolating sequence if and only if $\{z_n\}$ is a separated
sequence and there exists a constant $M>0$ such that
$$\sum_{z_n\in Q}(1-|z_n|)\leq M\ell(Q)$$
for any Carleson box $Q$, that is, a box $Q$ of the form
\begin{equation}\label{carleson_box}
Q=\left\{r e^{i\theta}:\ 0<1-r<\ell(Q),\
|\theta-\theta_0|<\ell(Q)\right\}.
\end{equation}
A sequence of points $\{z_n\}$ in the unit disc is called a
separated sequence if $\displaystyle\inf_{n\neq
m}\beta(z_n,z_m)>0$. A standard application of the open mapping
theorem tells that whenever $\{z_n\}$ is an interpolating sequence
there exists a constant $C=C(\{z_n\})>1$ with the following
property: for any bounded sequence $\{w_n\}$ there exists
$f\in\Hi$, such that $f(z_n)=w_n$ for $n=1,2,\dots$ with
$\|f\|_\infty\leq C\displaystyle\sup_{n}|w_n|$.

\medskip

\noindent The main purpose of this paper is to consider a
situation which is intermediate between these two classical
results.  On one hand, as in Carleson's Theorem, we want to
interpolate a concrete and natural target space of values
$\{w_n\}$.  On the other, as in the Nevanlinna-Pick interpolation
problem, we want to do it by functions in the unit ball $\B$ of
$\Hi$.

\bigskip

\noindent Our discussion starts with Schwarz's Lemma (see for
example \cite[p. 1]{Garnett}). If $f\in\B$ the classical Schwarz's
Lemma tells that
$$\beta(f(z),f(w))\leq\beta(z,w),\ \ \text{for any }z,w\in\D.$$
So, for any sequence of points $\{z_n\}$ in the unit disc, the
corresponding values $w_n=f(z_n)$, $n=1,2,\dots$, satisfy
$\beta(w_n,w_m)\leq\beta(z_n,z_m)$, for $n,m=1,2\dots$ However,
given a sequence of points $\{z_n\}\subset\D$ we can not expect to
interpolate any sequence of values $\{w_n\}$ satisfying the above
compatibility condition unless $\{z_n\}$ reduces to two points.
Actually having equality in Schwarz's Lemma for two different
points forces the function $f\in\B$ to be an automorphism and
hence we can not expect to interpolate any further value.  In
other words, the trace space arising from Schwarz's Lemma is too
large and we are lead to the following notion.

\medskip

\noindent A sequence of distinct points $\{z_n\}$ in the unit disc
will be called an \emph{interpolating sequence for $\B$} if there
exists a constant $\varepsilon=\varepsilon (\{z_n\})>0$ such that
for any sequence of values $\{w_n\}\subset\D$ satisfying the
compatibility condition
\begin{equation}\label{compatibility_condition}
\beta(w_m,w_n)\leq \varepsilon \beta(z_m,z_n),\ \ \ n,m=1,2,\dots,
\end{equation}
there exists a function $f\in\B$ such that $f(z_n)=w_n,\
n=1,2,\dots$

\medskip

\noindent Observe that this notion is conformally invariant, that
is, if $\{z_n\}$ is an interpolating sequence for $\B$ then so is
$\{\tau(z_n)\}$ for any automorphism $\tau$ of the unit disc.
Moreover the constant in the definition verifies
$\varepsilon(\{\tau(z_n)\})=\varepsilon(\{z_n\})$. It is obvious
that if a separated sequence is an interpolating sequence for $\B$
it is also an interpolating sequence for $\Hi$. As we will see,
the converse is far from being true.

\medskip

\noindent Let $\Delta$ denote a hyperbolic disc in $\D$, that is,
$\Delta=\{w:\beta(w,z)<\rho\}$ for some $z\in\D$ and $\rho>0$. Let
also $A_h(\Delta)$ denote the hyperbolic area of $\Delta$. The
main result of the paper is the following geometric description of
interpolating sequences for $\B$.

\begin{tma}\label{main_theorem}
A sequence $\{z_n\}$ of distinct points in the unit disc is an
interpolating sequence for $\B$ if and only if the following two
conditions hold:
\begin{itemize}
\item[$(a)$] $\{z_n\}$ is the union of two separated sequences.
\item[$(b)$] There exist constants $M>0$ and $0<\alpha<1$ such
that for any hyperbolic disc $\Delta$ with $A_h(\Delta)\geq M$ we
have
\begin{equation*}
\#\{z_k : z_k\in\Delta\}\leq A_h(\Delta)^{\alpha}.
\end{equation*}
\end{itemize}
\end{tma}
\noindent The density condition $(b)$ is the essential one. Let us
discuss its geometrical meaning. If the sequence $\{z_n\}$ was
merely separated we would have that there exists a constant $M>0$
such that for any hyperbolic disc $\Delta$, the estimate
$$\#\left\{z_k : z_k\in\Delta\right\}\leq M (A_h(\Delta)+1)$$
holds.  Hence interpolating sequences for $\B$ are exponentially
more sparse than separated sequences.  The density condition $b)$
can also be stated in the following way: there exist constants
$M>0$ and $0<\alpha<1$ such that for any Carleson box of the form
(\ref{carleson_box}) we have
\begin{equation}\label{density_condition}
\#\left\{z_k\in Q :
2^{-n-1}\ell(Q)<1-|z_k|\leq2^{-n}\ell(Q)\right\}\leq M2^{\alpha n}
\end{equation}
for any $n=1,2,\dots$

\medskip

\noindent Let us briefly describe the connection of this result
with the characterization of interpolating sequences in the Bloch
space. An analytic function f in $\D$ is in the Bloch space if
$$\sup(1-|z|^2)|f'(z)|<\infty,$$
where the supremum is taken over all $z\in\D$. A sequence of
points $\{z_n\}$ in the unit disc is an interpolating sequence for
the Bloch space if, whenever a corresponding sequence of complex
numbers $\{w_n\}$ satisfies
$$|w_n-w_n|\leq C\beta(z_n,z_m),\ \ n,m=1,2\dots$$
for some constant $C>0$, there exists a function $f$ in the Bloch
space solving the interpolation problem $f(z_n)=w_n$, for
$n=1,2,\dots$ It was proved in \cite{BN} that the conditions in
Theorem \ref{main_theorem} also characterize interpolating
sequences for the Bloch space. The nice book by K.\ Seip
\cite{Seip} contains a discussion of this result, its relations to
interpolation in other spaces, as well as several conditions
equivalent to the density condition $(b)$. Although the proof of
Theorem 1 contains many of the ideas of \cite{BN}, it is worth
mentioning that we can not deduce our result from the one for
Bloch functions. Roughly speaking, euclidean distance and the
interpolation constant $C$ in the case of the Bloch space are
replaced in this work by hyperbolic distance and the interpolation
constant $\varepsilon$ in (\ref{compatibility_condition}). So the
problem we treat is analogue to the one in \cite{BN}. However the
fact that hyperbolic distance now appears in both sides of
(\ref{compatibility_condition}), as well as that we are in a non
linear setting, makes the problem more difficult.

\bigskip

\noindent Let us now explain the main ideas in the proof. The
necessity is proven by taking values $\{w_n\}$ for which
$\beta(w_n,w_m)/\beta(z_n,z_m)$ is maximal and applying standard
techniques involving the non-tangential maximal function. The
proof of the sufficiency is considerably harder. Given a sequence
of points $\{z_n\}$ satisfying both conditions $(a)$ and $(b)$ and
a sequence of values $\{w_n\}\subset\D$ with
$\beta(w_n,w_m)\leq\varepsilon\beta(z_n,z_m)$ we have to find a
function $f\in\B$ with $f(z_n)=w_n$.  The main step of the proof
is the construction of a non-analytic function $\varphi$ in the
unit disc with $\varphi(z_n)=w_n$ such that
\begin{equation}\label{carleson_condition}
\int_Q\frac{|\nabla\varphi(z)|}{1-|\varphi(z)|^2}dA(z)\leq
C\varepsilon\ell(Q),
\end{equation}
for any Carleson box $Q$ of the form (\ref{carleson_box}).  Once
this is done, standard techniques involving $\Bm$ and
$\bar\partial$-equations provide a solution of the interpolation
problem in the unit ball of $\Hi$. The construction of the
function $\varphi$ is made in two different steps.  Using a
certain collection of dyadic Carleson boxes, we first construct a
non-analytic interpolating function $\varphi_0$.  It is more
convenient to work in the upper half plane $\R^2_+$ than in the
unit disc $\D$. Let us assume that $\{z_n\}$ is contained in the
unit square $[0,1)^2$. Let $I^{0}=[0,1)$ be the unit interval and
for $n=1,2,\dots$ consider the $2^n$ dyadic intervals
$I^n_j=[(j-1)/2^n,j/2^n)$ with $j=1,\dots,2^n$. Given an interval
$I$ of the line, let $Q(I)=\{x+iy\in\R^2_+:x\in I, 0<y\leq|I|\}$
be its associated Carleson box and $T(I)=\{x+iy\in Q(I):y>|I|/2\}$
its top part. We will also denote by $z(I)$ the center of $T(I)$,
that is, if $I=[a,b)$ then $z(I)=(a+b)/2+i3(b-a)/4$.

\medskip

\noindent Let $\A$ be the collection of dyadic intervals $I$ such
that $T(I)\cap\{z_n\}\neq\emptyset$. For the sake of simplicity,
let us assume that $\{z_n\}$ is a separated sequence and in fact,
that each $T(I)$ contains at most one single point of the sequence
$\{z_n\}$. The construction of the function $\varphi_0$ is based
on an useful combinatorial result which is proved in \cite{BN}. It
consists on considering a bigger family of dyadic intervals
$\G\supseteq\A$ and the function $\varphi_0$ will be constructed
by, roughly speaking, moving in the vertical direction at most
$\varepsilon$ hyperbolic units in each $T(I)$ for $I\in\G$ . The
family $\G$ verifies two properties that point to two opposite
directions. On the one hand the family $\G$ should be large to
guarantee that $\varphi_0$ reaches the corresponding values $w_n$
in the points $z_n\in T(I)$, $I\in\A$, as soon as these values
satisfy the compatibility condition
(\ref{compatibility_condition}). On the other hand, since
$\partial_y\varphi_0(z)$ will vanish in $\R^2_+\setminus\cup
T(I)$, where the union is taken over all $I\in\G$, the family $\G$
should be small enough to guarantee
$$\int_Q\frac{|\partial_y\varphi_0(z)|}{1-|\varphi_0(z)|^2}dA(z)\leq C\varepsilon\ell(Q)$$
for any Carleson box $Q\subset\R^2_+$. However, since there is no
control on the jumps of the function $\varphi_0$ on the vertical
sides of $Q(I)$, $I\in\G$, we can not expect that
$\partial_x\varphi_0$ satisfies an analogue estimate. To overcome
this difficulty, in the euclidean setting we would produce a
smooth interpolating function by averaging the functions
$\varphi_t$ corresponding to different sequences
$\{z_n+t:n=1,2,\dots\},$ $t\in[0,1)$, that is, by taking
$$\varphi(z)=\int_0^1\varphi_t(z+t)dt,\ \ z\in\R^2_+,$$
where $\varphi_t$ are constructed as explained above and verify
$\varphi_t(z_n+t)=w_n$ (see \cite[p. 76]{Seip}). However, in our
hyperbolic setting this does not make sense and the averaging
procedure is much more subtle.

\medskip

\noindent Given a point $z\in\R^2_+$, we would like to define
$\varphi(z)$ as the \emph{center of mass} of the set of points
$\{\varphi_t(z+t) : t \in [0,1)\}$, suitably weighted, in such a
way that for any pair of points $z,w\in\R^2_+$ the following
inequality holds
\begin{equation}\label{integral_ineq2}
\beta(\varphi(z),\varphi(w))\leq\int_0^1\beta(\varphi_t(z+t),\varphi_t(w+t))dt.
\end{equation}
This inequality and standard arguments will lead to estimate
(\ref{carleson_condition}). There are several possible notions of
\emph{center of mass} in hyperbolic space, and which one is
preferable depends on our specific purpose. Let us discuss briefly
two of them. One possibility consists in defining the center of
mass of a finite number of point masses inductively. The main
difficulty of this approach is to determine the representative
mass of the center of the masses (the naive attempt of taking the
sum of the masses lead to a definition that depends on the
partition used in the inductive step). See \cite{Gal} for a full
exposition of this. With this definition, the center of mass of
three points with equal mass coincides with the barycenter of the
triangle that they define; this fact can be used to prove that
this definition can not yield an estimate of the form of
(\ref{integral_ineq2}). Another approach is to define the center
of mass of a finite measure $\mu$ as the unique minimum of the
function
$$H(x)=\int_{\mathbb{R}_+^2} \beta(x,y)^2 d\mu(y).$$
As we will see in Section \ref{sect_center}, with this definition,
if we take $\mu$ as the pushforward measure of the Lebesgue
measure on $[0,1)$ by the map $t\mapsto \varphi_t(z+t)$, which
simply corresponds to assign a weight to the points
$\varphi_t(z+t)$, then inequality (\ref{integral_ineq2}) is
satisfied. The proof of this fact is a straightforward
generalization of a similar result stated in \cite{LPS}.

\bigskip

\noindent The paper is organized as follows.  The necessity in
Theorem \ref{main_theorem} is proved in Section
\ref{sect_necessity}. Section \ref{sect_center} is devoted to the
construction of the suitable center of mass of a certain measure
in the hyperbolic space. Since these notions have been considered
in the literature in general spaces of negative curvature, we will
present the results in the general context of Hadamard spaces.
This construction holds in any space of negative curvature.
Section \ref{sect_sufficiency} contains the proof of the
sufficiency in Theorem \ref{main_theorem}, and is the most
technical part of the paper.

\medskip

\noindent The letters $C,C_1,C_2,\dots$ will denote absolute
constants while $C(\delta)$ will denote a constant depending on
$\delta$.

\bigskip

\noindent It is a pleasure to thank Eero Saksman and María José
González for many helpful discussions. We are also grateful to the
referee for his/her careful reading of the paper.

\section{Necessity.}\label{sect_necessity}

\noindent We use the following normalization of the hyperbolic
distance
$$\beta(z,w)=\log_2\frac{1+|\frac{z-w}{1-\overline{w}z}|}{1-|\frac{z-w}{1-\overline{w}z}|},$$
where $z,w\in\D$, because it fits conveniently with the dyadic
decomposition of the disc, which will be a basic tool.  Let
$$I^{n}_j=[e^{i\pi(j-1)/2^n},e^{i\pi j/2^n}),\ \ \text{with }n\in\N\ \text{and }j=1,\dots,2^n,$$
be the standard collection of dyadic arcs on the unit circle $\T$
so that $|I^n_j|=2^{-n}$, where $|\cdot|$ is the normalized linear
measure on $\T$.  Given a dyadic arc $I\subset\T$ the
corresponding $Q(I)$ is called a dyadic Carleson square, and
$\{Q(I^n_j):j=1,\dots,2^n,n\in\N\}$ is the dyadic decomposition of
$\D$. Given a dyadic arc $I_j^n$ we will say that
$z(I_j^n)=(1-|I^n_j|)\exp(i\pi(j+1)/2^{n+1})$ is the center of the
Carleson square $Q(I^n_j)$. It is easy to deduce that if $I,J$ are
dyadic arcs, $I\subseteq J$ and $|I|=2^{-k}|J|$, for some
$k\in\N$, then $|\beta(z(I),z(J))-k|\leq C$, where $C$ is a
universal constant independent of $I$, $J$ and $k$.

\medskip

\noindent First of all let us show the equivalence between
condition $(b)$ in Theorem \ref{main_theorem} and
(\ref{density_condition}). Assume $(b)$ is verified, by conformal
invariance we may take $Q=\D$ in (\ref{density_condition}). Now
$$A_h(\Delta)=C\int_\Delta\frac{1}{(1-|z|^2)^2}dA(z),$$
where $C>0$ is a universal constant and $dA(z)$ is the euclidean
area measure (see \cite{Anderson}). Then it is easy to show that
for any $n=1,2,\dots$ we have that $A_h(D(0,1-2^{-n}))=C_12^{n}$,
where $C_1$ is a universal constant. So condition $(b)$ implies
that there exists $M_1>0$ such that
$$\#\left\{z_k\in \D :
2^{-n-1}<1-|z_k|\leq2^{-n}\right\}\leq M_12^{\alpha n}.$$
Conversely, assume (\ref{density_condition}) holds and let
$\Delta$ be an euclidean disc centered at the origin and with
euclidean radius $r<1$. Then $A_h(\Delta)=C_2(1-r)^{-1}$, where
$C_2$ is an absolute constant. Pick $n\in\N$ such that
$2^{-n-1}<1-r\leq 2^{-n}$, then applying (\ref{density_condition})
to the sets $\{z_k\in\D:2^{-j-1}<1-|z_k|\leq 2^{-j}\}$ for
$j=0,1,\dots,n$ and summing in $j$ we deduce that
$\#\{z_k:z_k\in\Delta\}\leq M_2 2^{\alpha n}\leq
M_3A_h(\Delta)^{\alpha}$.

\medskip

\noindent With the normalization of the hyperbolic distance given
above, we may prove that estimate (\ref{density_condition}), and
equivalently condition $(b)$ in Theorem \ref{main_theorem}, can be
expressed in the following way: There exist constants $M>0$ and
$0<\alpha<1$ such that for any $z\in\D$
\begin{equation}\label{density_condition2}
\#\{z_k\ :\ \beta(z,z_k)\leq n\}\leq M2^{\alpha n},\ \ n=1,2,\dots
\end{equation}

\subsection{Union of two separated sequences}\label{sect_neces_2seq}

We start with the easiest part of the necessity which is the
separation condition $(a)$ in Theorem \ref{main_theorem}.  This
condition appears because our target space is defined in terms of
first differences, while Cauchy's formula tells that
$(1-|z|)^n|f^{(n)}(z)|\leq C(n)\|f\|_\infty$ for any $z\in\D,\
n=1,2,...$  So if three points of the sequence $\{z_n\}$ were in a
small hyperbolic disc, the corresponding values should satisfy a
more restrictive smoothness condition which could be expressed in
terms of second differences. More concretely, we will show that
there exists $\delta>0$ such that any hyperbolic disc of radius
$\delta$ has at most two points of the sequence.  Let
$z_1,z_2,z_3$ be three points of the sequence $\{z_n\}$ with
$$\max\{\beta(z_1,z_2),\beta(z_2,z_3)\}\leq\beta(z_1,z_3)=\rho.$$
We will show that $\rho$ is bounded below.  By conformal
invariance we may assume $z_1=0$.  Now take the values $w_1=w_3=0$
and $w_2=\varepsilon\beta(z_1,z_2)$, where
$\varepsilon=\varepsilon(\{z_n\})$ is the interpolation constant,
that is, the quantifier appearing in
(\ref{compatibility_condition}). It is clear that these values
satisfy the compatibility condition
(\ref{compatibility_condition}), so then there exists a function
$f\in\B$ with $f(0)=f(z_3)=0$ and
$f(z_2)=\varepsilon\beta(0,z_2)$. Hence there is a point $\zeta$
in the radius from 0 to $z_2$ with
$(1-|\zeta|)|f'(\zeta)|>C_1\varepsilon$.  On the other hand since
$$|f(z)|\leq|z|\left|\frac{z-z_3}{1-\bar z_3z}\right|,\ \ z\in\D,$$
we have that
$$|f'(z)|\leq C_2\rho\ \ \text{if }|z|<\rho.$$
So we deduce that there exists a constant $C_3>0$ such that
$\rho\geq C_3\varepsilon$. Hence we have proved that $\{z_n\}$ can
be written as $\{z_n\}=\{z^{(1)}_n\}\cup\{z_n^{(2)}\}$, where both
sequences are separated.


\subsection{Density condition}\label{sect_neces_density}

\subsubsection{Some lemmas}

The proof of the density condition $(b)$ is based on the following
two auxiliary results.  The first one is a convenient version of a
well known estimate of the non-tangential maximal function. The
second one is an elementary combinatorial statement. Fix $M>1$.
Given a set $E\subseteq\D$, let $\Pi(E)=\Pi_M(E)$ denote the set
of points $\xi\in\T$ such that the Stolz angle
$\Gamma_M(\xi)=\{z\in\D:|z-\xi|<M(1-|z|)\}$ intersects $E$.

\begin{lma}\label{projection}
There is a constant $C(M)>0$ such that for any $f\in\B$ with
$f(0)=0$ and any $\eta>0$ we have
$$\big|\Pi\left(\left\{z\in\D:|f(z)-1|<\eta\right\}\right)\big|\leq C\eta.$$
\end{lma}

\begin{proof}
Consider the function $g=(1-f)^{-1}$ which maps the disc into the
right half plane.  Let $\Mx g$ be the non-tangential maximal
function of $g$, that is
$$\Mx g(\xi)=\sup\big\{|g(z)|:z\in\Gamma(\xi),\xi\in\partial\D\big\},$$
where $\Gamma(\xi)$ is the Stolz angle with vertex at $\xi$. Since
$g$ has positive real part, then $\Mx g$ satisfies the weak type
estimate
$$\left|\left\{\xi\in\T : |\Mx g(\xi)|>\lambda\right\}\right|\leq\frac{C}{\lambda}$$
for any $\lambda>0$.  Here $C$ is a universal constant independent
of $g$. Since
$$\Pi\left(\left\{z\in\D:|f(z)-1|<\eta\right\}\right)=\left\{\xi\in\T : \Mx g(\xi)>\frac{1}{\eta}\right\}$$
the proof is completed.
\end{proof}

\begin{lma}\label{sum_good_squares}
Let $M>0$ and $0<\alpha<1$ be fixed constants.  Let $\A$ be a
collection of dyadic arcs of the unit circle $\T$. Assume that for
any dyadic arc $I$ and any positive integer $n$ at least one of
its two halves, which we denote by $\widetilde{I}$, satisfies
\begin{equation}\label{good_square_estimate}
\#\big\{J\in\A : J\subset\widetilde{I},\ |J|=2^{-n}|I|\big\}\leq M
2^{n\alpha}.
\end{equation} Then for any dyadic arc
$I$ and any positive integer $n$ we have
$$\#\big\{J\in\A : J\subset I,\ |J|=2^{-n}|I|\big\}\leq \frac{2M}{1-2^{-\alpha}} 2^{n\alpha}.$$
\end{lma}

\begin{proof}
Fix a dyadic arc $I$. By hypothesis at least one of its two
halves, denoted by $I_1$, verifies estimate
(\ref{good_square_estimate}). If the other half also verifies the
estimate we denote it by $I_2$ and we stop the decomposition. If
not, by hypothesis at least one of its two halves, denoted now by
$I_2$, must verify that
\begin{equation*}
\#\big\{J\in\A : J\subset I_2,\ |J|=2^{-n}|I|\big\}\leq M
2^{(n-1)\alpha}.
\end{equation*}
Repeating this process at most $n$ times we cover $I$ by $m\leq
n+1$ pairwise disjoint dyadic intervals $\{I_j:j=1,\dots,m\}$ with
$|I_j|=2^{-j}|I|$ if $j<m$ and $|I_m|=2^{1-m}|I|$ if $m\leq n$, or
$|I_m|=2^{-n}|I|$ if $m=n+1$. Also these intervals satisfy that
$$\#\left\{J\in\A:J\subset I_j,\ |J|=2^{-n}|I|\right\}\leq 2M2^{(n-j)\alpha},\ \ j=1,\dots,m.$$
Hence,
$$\#\left\{J\in\A:J\subset I,\ |J|=2^{-n}|I|\right\}\leq 2M\sum_{j=1}^m2^{(n-j)\alpha}.$$
\end{proof}

\subsubsection{Necessity of the density condition.}
Let $\{z_n\}$ be an interpolating sequence for $\B$ and let
$\varepsilon=\varepsilon(\{z_n\})>0$ be the interpolation constant
appearing in (\ref{compatibility_condition}). We first prove that
there exist constants $M=M(\varepsilon)>0$ and
$\alpha=\alpha(\varepsilon)$ with $0<\alpha<1$ such that for any
positive integer number $n$ and any dyadic arc $I\subset\T$ at
least one of its two halves, say $\tilde{I}$, satisfies
\begin{equation}\label{good_square}
\#\left\{z_k\in Q(\tilde{I}):n-1<\beta(z(I),z_k)\leq n\right\}\leq
M2^{n\alpha}.
\end{equation}
By conformal invariance we may assume that $I=\T$ and $z(I)=0$.
Let $\gamma>0$ be a small number to be fixed later with
$\gamma<\varepsilon$.  Now let $\F_m=\F_m(n)$, $m=1,2$, be the two
collections of points of $\{z_n\}$ given by
\begin{align*}
\F_1&=\left\{z_k\ :\ n-1<\beta(0,z_k)\leq n,\
|\operatorname{arg}z_k|<\frac{\pi}{2}-2^{-n\gamma}\right\},\\
\F_2&=\left\{z_k\ :\ n-1<\beta(0,z_k)\leq n,\
|\operatorname{arg}z_k-\pi|<\frac{\pi}{2}-2^{-n\gamma}\right\}.
\end{align*}

\noindent See Figure \ref{disc}.
\begin{figure}[htbp]

  \psfrag{o}[c][c]{$0$}
  \psfrag{r}[c][c]{$1-2^{-n}$}
  \psfrag{D}[c][c]{$\D$}

  \psfrag{l}[c][c]{$2\cdot 2^{-n\gamma}$}

  \begin{center}

   \includegraphics[width=.45\textwidth]{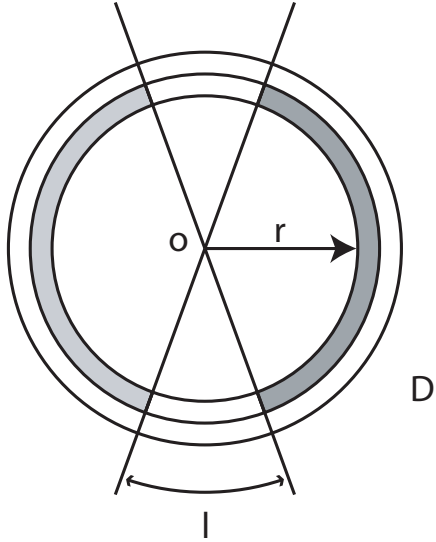}

   \caption{The family $\F_1$ are the points of the sequence in the dark grey zone, while the family $\F_2$ are the ones in the light grey zone.}
   \label{disc}

\end{center}
\end{figure}
Notice that if $z\in\F_1$ and $\tilde{z}\in\F_2$ we have
$\beta(z,\tilde{z})\geq Cn\gamma$, where $C$ is an absolute
constant.  Now consider the values $\{w_k\}$ given by
\begin{align*}
w_k&=\hspace{0.32cm}1-2^{-C\varepsilon\gamma n}\ \ \text{if }z_k\in\F_1,\\
w_k&=-1+2^{-C\varepsilon\gamma n}\ \ \text{if }z_k\in\F_2.
\end{align*}
It is easy to show that the compatibility condition
(\ref{compatibility_condition}) holds, that is,
$\beta(w_j,w_k)\leq\varepsilon\beta(z_j,z_k)$ if
$z_j,z_k\in\F_1\cup\F_2$.  Hence by hypothesis there exists a
function $f\in\B$ with $f(z_k)=w_k$ for any $z_k\in\F_1\cup\F_2$.
Notice that there exists at least one index $m=1,2$ for which
$|\Rea f(z_k)-\Rea f(0)|>1-2^{-C\varepsilon\gamma n}$ for any
$z_k\in\F_m$. So let us assume $m=1$. Write
$w=1-2^{-C\varepsilon\gamma n}$ and let $\tau$ be the automorphism
of the disc which maps $f(0)$ to the origin and which satisfies
$\operatorname{arg}\tau(w)=0$. Observe that $C_1\varepsilon\gamma
n\leq\beta(f(0),w)=\beta(0,\tau(w))$. Now take $h=\tau\circ f$ and
notice that
$$\beta(0,h(z_k))\geq C_1\varepsilon\gamma n,\ \ z_k\in\F_1.$$
Since $\operatorname{arg}h(z_k)=0$, we have
$$|1-h(z_k)|\leq C_2 2^{-C_1\varepsilon\gamma n},\ \ z_k\in\F_1.$$
Now we can apply Lemma \ref{projection} to the function $h$ and
the parameter $\eta=C_2 2^{-C_1\varepsilon\gamma n}$ to deduce
that
$$|\Pi(\F_1)|\leq C_3 2^{-C_1\varepsilon\gamma n}.$$
Notice that the projection $\Pi$ of a single point of $\F_1$ is an
arc of length comparable to $2^{-n}$.  Since the sequence
$\{z_n\}$ is the union of two separated sequences, the
corresponding intervals $\{\Pi(z_n)\}_n$ are quasidisjoint, that
is, the function
$$\sum_{n}\chi_{\Pi(z_n)}(\xi)$$
is bounded on the unit circle. Here $\chi_E$ represents the
characteristic function of a set $E$. So we deduce
$$2^{-n}\#\F_1\leq C_42^{-C_1\varepsilon\gamma n}.$$
Moreover there are at most $C_5 2^{n(1-\gamma)}$ points such that
$n-1<\beta(0,z_k)\leq n$ and
$\pi/2-2^{-n\gamma}\leq|\operatorname{arg}z_k|\leq\pi/2$. So now
the estimate (\ref{good_square}) follows because the term on the
left hand side is bounded by
$C_52^{n(1-\gamma)}+C_42^{n(1-C_1\varepsilon\gamma)}$. We only
need to choose $\gamma<1$ so that $C_1\varepsilon\gamma<1$ and
pick $\alpha=\max\{1-\gamma, 1-C_1\varepsilon\gamma\}$ and
$M=2\max\{C_4,C_5\}$.

\bigskip

\noindent We now prove the necessity of the density condition
$(b)$. So, given a point $z\in\D$ we wish to estimate the number
of points $z_k$ in the sequence $\{z_n\}$ such that
$\beta(z_k,z)\leq n$. Equivalently, according to
(\ref{density_condition}), given a Carleson box $Q$ we need to
estimate $\#\F_n$, where $\F_n=\big\{z_k\in Q :
2^{-n-1}\ell(Q)\leq 1-|z_k|\leq 2^{-n}\ell(Q)\big\}$.  Since any
arc of the circle is contained in the union of at most four dyadic
arcs of comparable total length, we can assume that
$\overline{Q}\cap\T$ is a dyadic arc.  Denoting by $\A_n$ the set
of the dyadic arcs $I$ of length $2^{-n}$ such that
$Q(I)\cap\F_n\neq\emptyset$, we have $\#\F_n\leq C\#\A_n$, where
$C$ is a constant only depending on $\{z_n\}$. An application of
Lemma \ref{sum_good_squares} will provide an estimate of $\#\A_n$.
For this, we only have to observe that estimate
(\ref{good_square}) gives the hypothesis
(\ref{good_square_estimate}) in Lemma \ref{sum_good_squares}. So
estimate (\ref{density_condition}) is proved.

\section{Center of Mass.}\label{sect_center}

\noindent Let $\mu$ be a probability measure on $\R^n$ with
bounded support. We may define the center of mass of $\mu$ as the
unique point in $\R^n$ that minimizes the function
$$H_\mu(x)=\int_{\R^n} \|x-y\|^2d\mu(y).$$
The fact that guarantees the existence and uniqueness of that
minimum is the strong convexity of the euclidean distance (see
\cite[p. 332]{Bur}). This means that the same construction can be
performed in a more general setting, namely, in a metric space
(with certain regularity) such that the distance is strongly
convex. For instance, we can define it in hyperbolic space.

\medskip

\noindent Although the results of this section will be applied
only to the case of hyperbolic plane, we are going to work in the
more general framework of Hadamard spaces. There are at least two
reasons for doing that. The first one is that the restriction to
hyperbolic plane seems not to simplify the arguments. The second
reason is that we are going to follow a small part of the work
done in \cite{LPS}, where the results are placed in this broader
context. Readers not interested in this generality may skip the
definitions given below concerning Hadamard spaces, and just
replace \emph{Hadamard space} by \emph{hyperbolic space} in the
statements.

\medskip

\noindent The aim of this section is to define the center of mass
of a probability measure in a Hadamard space and to prove the
contractive property stated in Proposition \ref{convex_ineq}
below. Let us start with some definitions.

\bigskip

\noindent A \emph{geodesic space} is a metric space in which any
two points can be connected by a minimizing geodesic (see \cite[p.
31]{Heinonen}). Next, we want to give the definition of
nonpositive curvature in the sense of Alexandrov. Let $(X,d)$ be a
geodesic space and $\Gamma$ a triangle in $X$ formed by minimizing
geodesics with vertices $A_1,A_2,A_3$. Consider a euclidean
triangle $\Gamma'$ with vertices $A_1',A_2',A_3'$ such that
$d(A_i',A_{j}')=d(A_{i},A_{j})$ for $i,j=1,2,3$. Let $P$ be a
point in the segment $\overline{A_2A_3}$, and $P'$ the point in
the segment $\overline{A_2'A_3'}$ such that $d(A_2,P)=
d(A_2',P')$. The geodesic space $(X,d)$ is said to have
\emph{nonpositive curvature} if for every $x\in X$ there exists
$r_x>0$ such that every triangle $\Gamma$ contained in $B(x,r_x)$
satisfies, with the above notation, the following inequality
$$d(A_1,P)\leq d(A_1',P').$$

\noindent If we parameterize the segment $\overline{A_2A_3}$ by a
geodesic $\sigma: [0,1]\rightarrow X$ such that $\sigma(0)=A_2$
and $\sigma(1)=A_3$, the above inequality turns out to be
\begin{equation}
\label{conv_dist}
 d(A_1,\sigma(t))^2\leq (1-t)d(A_1,A_2)^2+td(A_1,A_3)^2-t(t-1)d(A_2,A_3)^2,
\end{equation}
for all $t\in [0,1]$, as an easy computation in euclidean geometry
shows. Note that inequality (\ref{conv_dist}) says that the
function $f(x)=d(A,x)^2$, when restricted to a minimizing
geodesic, is strongly convex (see \cite[p. 332]{Bur}).

\noindent A complete simply connected space of nonpositive
curvature is called a \emph{Hadamard space}.

\begin{obs}
Using hyperbolic trigonometry we can verify that hyperbolic space
has indeed nonpositive curvature, and consequently it is a
Hadamard space. More generally, it is a theorem by E.\ Cartan
\cite{Car} that a Riemannian manifold of sectional curvature
bounded above by $0$ is a space of nonpositive curvature, in the
sense of Alexandrov.
\end{obs}

\noindent An important fact about Hadamard spaces is that the
nonpositive curvature condition holds for every triangle. In other
words, for every point $x$ the constant $r_x$ that appears in the
definition of nonpositive curvature can be taken arbitrarily large
(see \cite[p. 329]{Bur}). The following results have been taken
almost verbatim from \cite{LPS}, (Lemmas 2.1, 4.1 and 4.2).
\begin{lma}
Let $\mu$ be a probability measure of bounded support on a
Hadamard space $(X,d)$. Then the function defined by
$$H_\mu(x)=\int_X d(x,y)^2 d\mu(y),$$
has a unique minimum. Moreover, if we denote this point by
$c_\mu$, for all $z \in X$ we have
\begin{equation}
H_\mu(z)\geq H_\mu(c_\mu)+d(c_\mu,z)^2. \label{des_cent}
\end{equation}
\end{lma}

\begin{proof}
First of all note that from inequality (\ref{conv_dist}) and the
fact that $\mu(X)=1$, we have
$$H_\mu(\sigma(t))\leq (1-t)H_\mu(z)+tH_\mu(z')-t(1-t)d(z,z')^2,$$
where $z,z'\in X$ and $\sigma:[0,1]\rightarrow X$ is a geodesic
joining them starting at $z$. From the above inequality, the
uniqueness of the minimum of $H_\mu$ follows immediately. For the
existence, let $m$ be the infimum of $H_\mu$ and $\{z_i\}$ a
minimizing sequence. For $i,j$ sufficiently large,
$H_\mu(z_i),H_\mu(z_j)<m+(\epsilon/2)^2$. Now, taking $t=1/2$ in
the above inequality, we get
$$m\leq m+(\epsilon/2)^2 -\frac{1}{4}d(z_i,z_j)^2,$$
which shows that $\{z_i\}$ is a Cauchy sequence. Since $X$ is
complete this sequence converges to a point. This proves that the
minimum of $H_\mu$ is attained at some point which will be denoted
by $c_\mu$. Now, let $z$ be a point in $X$. Taking a geodesic from
$z$ to $c_\mu$, from the above inequality we obtain
$$H_\mu(c_\mu)\leq (1-t)H_\mu(c_\mu)+tH_\mu(z)-t(1-t)d(z,c_\mu)^2,$$
for $t\in[0,1]$. It yields
$$H_\mu(c_\mu)\leq H_\mu(z)-(1-t)d(z,c_\mu)^2,$$
and inequality (\ref{des_cent}) follows. (Note that the same
argument can be applied to an arbitrary strongly convex function
bounded from below, see \cite[p. 333]{Bur}.)
\end{proof}

\smallskip

\noindent The point $c_\mu$ is called the \emph{center of mass} of
$\mu$.

\medskip

\noindent In order to prove the main result of this section, we
need the following instance of Reshetnyak's quadrilateral
inequality, see Lemma 2.1 of \cite{LPS}.
\begin{lemma}
Let $(Y,d)$ be a Hadamard space and $y,y',z,z'\in Y$. Then
$$d(y,z')^2+d(y',z)^2\leq d(y,z)^2 + d(y',z')^2+2d(y,y')d(z,z').$$
\label{Resh_ineq}
\end{lemma}

\noindent Let $(X,\mu)$ be a measure space of finite measure, and
$(Y,d)$ be a Hadamard space. Given a bounded map
$f:X\longrightarrow Y$ we denote by $c_f$ the center of mass of
the push-forward measure $f_*\mu$ defined as $(f_*
\mu)(E)=\mu(f^{-1}(E))$, for any measurable set $E\subset Y$.

\begin{prop}\label{convex_ineq}
Let $(X,\mu)$ be a measure space with $\mu(X)=1$ and let $(Y,d)$
be a Hadamard space with distance $d$. If $f,g\in \Li (X,Y)$ are
two essentially bounded maps from $X$ to $Y$, then
$$d(c_f,c_g)\leq \int_X d(f(x),g(x))d\mu(x).$$
\end{prop}

\begin{proof}
Consider the quadrilateral defined by the points
$f(x),g(x),c_f,c_g$ and apply Lemma A to obtain
\begin{eqnarray*}
d(f(x),c_g)^2+d(g(x),c_f)^2&\leq& d(f(x),c_f)^2+d(g(x),c_g)^2 +\\
&& 2d(f(x),g(x))d(c_f,c_g).
\end{eqnarray*}
Integrating on $X$ and making a change of variables, we get
\begin{eqnarray*}
H_{f_*\mu}(c_g)+H_{g_*\mu}(c_f)&\leq &H_{f_*\mu}(c_f)+H_{g_*\mu}(c_g)+\\
&&2d(c_f,c_g)\int_X d(f(x),g(x))d\mu(x).
\end{eqnarray*}
On the other hand, inequality (\ref{des_cent}) yields
$$(H_{f_*\mu}(c_g)- H_{f_*\mu}(c_f))+(H_{g_*\mu}(c_f) -H_{g_*\mu}(c_g))\geq 2d(c_g,c_f)^2.$$
Combining these two last inequalities, we obtain
$$d(c_g,c_f)\leq \int_X d(f(x),g(x))d\mu(x).$$
\end{proof}

\section{Sufficiency.}\label{sect_sufficiency}
\noindent The proof of the sufficiency is presented in the upper
half plane $\R_+^2$.  We will also denote by $\beta(z,w)$ the
hyperbolic distance between two points $z,w\in\R^2_+$, that is,
$$\beta(z,w)=\frac{1}{2}\log_2\frac{1+\left|\frac{z-w}{z-\overline w}\right|}{1-\left|\frac{z-w}{z-\overline w}\right|}.$$
Recall also that a measure $\mu$ defined in the upper half plane
is a Carleson measure if there exists an absolute constant $C>0$
such that for any Carleson box $Q\subset\R_2^+$ we have that
$$\mu(Q)\leq C\ell(Q).$$
The Carleson norm of the measure $\|\mu\|_C$ is the infimum of the
constants $C>0$ for which this inequality holds.

\bigskip

\noindent Let $\{z_n\}$ be a sequence of points in $\R^2_+$
satisfying conditions $(a)$ and $(b)$ in Theorem
\ref{main_theorem}.  We want to find a constant
$\varepsilon=\varepsilon(\{z_n\})>0$ so that for any sequence of
values $\{w_n\}\subset\D$ for which the compatibility condition
$$\beta(w_n,w_m)\leq\varepsilon\beta(z_n,z_m)$$
holds, there exists a function $f\in\B$ with $f(z_n)=w_n$ for
$n=1,2,\dots$  We first assume that $\{z_n\}$ is a separated
sequence with constant of separation $\delta$. This is not the
general case, but contains the main ideas of the proof. By a
normal families argument we can assume that $\{z_n\}$ has only a
finite number of points. The main part of our argument is to
construct a non-analytic mapping $\varphi:\R^2_+\longrightarrow\D$
that verifies the following three conditions:
\begin{itemize}
\item[$(A)$] For $n=1,2\dots$ we have $\varphi\equiv w_n$ in
$\Delta(z_n,C\delta)$, where $C<1$ is a fixed positive constant
that will be chosen later.\vspace{0.3cm} \item[$(B)$] For any
$z,w\in\R^2_+$ we have $\beta(\varphi(z),\varphi(w))\leq
C\varepsilon\beta(z,w)$.\vspace{0.3cm} \item[$(C)$] The measure
$|\nabla\varphi(z)|dA(z)/(1-|\varphi(z)|^2)$ is a Carleson measure
with Carleson norm smaller than $C\varepsilon$.
\end{itemize}
The next three subsections are devoted to the construction of the
function $\varphi$ when $\{z_n\}$ is a separated sequence, while
in the last two sections $\bar{\partial}$-techniques are applied
to obtain analytic solutions of our interpolation problem.

\subsection{Reduction to well separated sequences}\label{sect_ext_phi}
\noindent The main purpose of this subsection is to reduce the
construction of the smooth interpolating function $\varphi$ to the
case when the sequence $\Lambda=\{z_n\}$ is separated, with large
constant of separation. Moreover we will show that we can assume
that the sequence $\Lambda$ consists of the centers of a certain
collection of dyadic Carleson squares.

\medskip

\noindent First of all we may assume that the whole sequence
$\{z_n\}$ is contained in $Q([0,1))$.  We also can add the point
$z_0=1/2+i3/2$ with value $w_0=0$ in the sequence $\{w_n\}$.

\medskip

\noindent So consider a separated sequence $\{z_n\}$. Given a
positive number $N>0$, we are going to construct a sequence
$\Lambda_0=\{z(n),\ n=0,1,\dots\}$ such that
$$\Lambda\subseteq\bigcup_{z(n)\in\Lambda_0}\overline{\Delta\left(z(n),N\right)}$$
and
$$\inf\big\{\beta(z(n),z(m)):z(n),z(m)\in\Lambda_0,\ n\neq m\big\}\geq N.$$

\smallskip

\noindent The construction of $\Lambda_0$ is as follows.  Assume
that the sequence $\Lambda$ is ordered such that $\Ima z_n\geq\Ima
z_{n+1},\ n=0,1,2\dots$ Take $z(0)=z_0\in\Lambda_0$, and $w(0)=0$.
We consider the family $\F_1$ of dyadic intervals so that $T(I)$
contains a point of $\Lambda\setminus\Delta(z(0),N)$.  Take $I_1$
so that $|I_1|=\max\{|I|:I\in\F_1\}$ and let $z(1)$ be the center
of $T(I_1)$. Take $w(1)=w_{j(1)}$ such that $z_{j(1)}$ is the
closest point of $\{z_n\}$ to $z(1)$.  Now let $\F_2$ be the
family of dyadic intervals $I$ so that $T(I)$ contains a point of
$\Lambda\setminus(\Delta(z(0),N)\cup\Delta(z(1),N))$.  We continue
this construction by induction.  Take
$\Lambda_0=\{z(0),z(1),\dots\}$.  Then $\Lambda_0$ satisfies the
conditions above.  For each $z(n)\in\Lambda_0$ choose
$w(n)=w_{j(n)}$ such that $z_{j(n)}$ is the closest point of
$\Lambda$ to $z(n)$.

\medskip

\noindent Now assume that we have constructed a smooth function
$\varphi_0:\R^2_+\longrightarrow\D$ such that interpolates the
corresponding values $\{w(n)\}$ for points in $\Lambda_0$ and
verifies properties $(A),(B),(C)$ above. Next we will construct a
smooth function which interpolates the prescribed values at the
sequence $\Lambda$ and satisfies analogous estimates.

\medskip

\noindent Since $\Lambda$ is separated, there exists $\delta>0$
such that the hyperbolic discs
$\big\{\Delta(z_n,2\delta):z_n\in\Lambda\big\}$ are pairwise
disjoint.  We define the function $\varphi$ as
$\varphi\equiv\varphi_0$ on $\R_+^2\setminus\cup
\Delta(z_n,2\delta)$, where the union is taken over all points
$z_n\in\Lambda$. For each point $z_n\in\Lambda$ we define
$\varphi\equiv w_n$ on $\Delta(z_n,\delta)$ and $(A)$ follows. So
$\varphi$ is now defined in $\R^2_+\setminus
\cup\{z:\delta\leq\beta(z_n,z)<2\delta\}$. A basic Lemma of
McShanne \cite{McS} and Valentine \cite{Val} (see \cite[p.
43]{Heinonen}) says that if $X$ is a metric space and $E\subset
X$, any Lipschitz function on $E$ can be extended to $X$ with the
same Lipschitz constant. The compatibility condition
(\ref{compatibility_condition}) and property $(B)$ for $\varphi_0$
give that we can extend $\varphi$ to $\Delta(z_n,2\delta)\setminus
\Delta(z_n,\delta)$ so that
\begin{equation*}
\beta(\varphi(z),\varphi(\tilde{z}))\leq
C(\delta)\varepsilon\beta(z,\tilde{z})
\end{equation*}
for any $z,\tilde{z}\in\R^2_+$.  So $(B)$ holds for $\varphi$.
Dividing by $\beta(z,\tilde{z})$ and taking $\tilde{z}\rightarrow
z$ we deduce that
\begin{equation*}
(\Ima z)|\nabla\varphi(z)|\leq C_1(\delta)\varepsilon
(1-|\varphi(z)|^2)
\end{equation*}
for any $z\in\R^2_+$.  So we have that
$$\dfrac{|\nabla\varphi(z)|}{(1-|\varphi(z)|^2)}\leq\dfrac{|\nabla\varphi_0(z)|}{(1-|\varphi_0(z)|^2)}+C_2(\delta)\varepsilon\sum_n\dfrac{\chi_{\Delta(z_n,2\delta)}(z)}{\Ima z}.$$
Since $\varphi_0$ verifies property $(C)$, we only have to show
that the second term gives rise to a Carleson measure.  So let $Q$
be a Carleson box. If $\Delta(z_n,2\delta)\cap Q\neq\emptyset$
notice that $z_n\in C(\delta)Q$. Hence
$$\sum_n\int_Q\frac{\chi_{\Delta(z_n,2\delta)}(z)}{\Ima z}dA(z)\leq C_3(\delta)\sum_{z_n\in C(\delta)Q}\Ima z_n.$$
Since the density condition $(b)$ in Theorem 1 implies Carleson's
condition, then this term is bounded by
$MC_3(\delta)C(\alpha)\ell(Q)$, so $\varphi$ verifies $(C)$.

\noindent Hence, without loss of generality, we may assume that
the sequence $\Lambda=\{z_n\}$ is a well separated sequence formed
by the centers of some dyadic Carleson squares.

\subsection{A non-smooth interpolating function}\label{sect_non-smooth_funct}

As explained before, given a sequence of values satisfying the
compatibility condition (\ref{compatibility_condition}), we wish
to construct a smooth interpolating function $\varphi$ satisfying
conditions $(A)$, $(B)$ and $(C)$. The main purpose of this
section is to construct a piecewise continuous interpolating
function $\varphi$ in the upper half plane which satisfies
condition $(A)$ and, roughly speaking, conditions $(B)$ and $(C)$
if we restrict attention to the vertical direction. These two
conditions will fail in the horizontal direction, but we will
still have certain control which will be used later. This is
better expressed in terms of collections of dyadic squares. For
this purpose, we will use one of the main auxiliary results from
\cite{BN}, a covering Lemma that comes in handy in the present
construction as well. As explained in the previous section, we may
assume that the points $\{z_n\}$ are centers of $T(I)$, for a
certain collection $\A$ of dyadic intervals. The density condition
(\ref{density_condition}) then translates to the following one:
for any dyadic interval $I$ we have that
\begin{equation}\label{density_intervals}
\#\big\{I_k\in\A : I_k\subset I, |I_k|=2^{-n}|I|\big\}\leq
M_12^{\alpha n},\ \ n=0,1,2\dots
\end{equation}
The covering Lemma is the following.

\begin{lemmaB}
Let $\A$ be a collection of dyadic intervals.  The following
conditions are equivalent:
\begin{itemize}
\item[$(a)$] There exist constants $M_1>0$ and $0<\alpha<1$ such
that for any dyadic interval $I$ and any $n=0,1,2,\dots$ we have
\begin{equation*}
\#\big\{I_k\in\A : I_k\subset I, |I_k|=2^{-n}|I|\big\}\leq
M_12^{\alpha n}
\end{equation*}
\item[$(b)$] There exist a family $\G$ of dyadic arcs with
$\A\subset\G$ and a constant $C>0$ such that the following two
conditions hold
\begin{itemize}
\item[$(b1)$]For any dyadic arc $J$ we have
$$\sum_{I\subseteq J, I\in\G}|I|\leq C|J|.$$
\item[$(b2)$]For any pair of intervals $I_0\in\A$ and $I_1\in\G$
with $I_0\subseteq I_1$ we have
$$\#\big\{I\in\G : I_0\subseteq I\subseteq I_1\big\}\geq C^{-1}\log\frac{|I_1|}{|I_0|}.$$
\end{itemize}
\end{itemize}
\end{lemmaB}
\noindent We will call $\G$ the \emph{intermediate family}
associated to $\A$. The properties $(b1)$ and $(b2)$ point to two
opposite directions. On the one hand the family $\G$ should be
large to guarantee $(b2)$ but on the other hand, $\G$ should be
small to guarantee $(b1)$. As it turns out, our density condition
$(b)$ in Theorem \ref{main_theorem} is what we need to arrive to a
compromise. The intermediate family will help us to construct the
non-smooth interpolating function, as we may see in the following
lemma.

\begin{lma}\label{nonsmooth_fnct}
Let $\A$ be a family of dyadic intervals satisfying condition
(\ref{density_intervals}), and let $\{z_n\}$ be the sequence of
centers $\{z(I):I\in\A\}$. Let $\G$ be the intermediate family
given by Lemma B. Let $\{w_n\}$ be a sequence of values in the
unit disc satisfying the compatibility condition
$$\beta(w_n,w_m)\leq\varepsilon\beta(z_n,z_m),\ n,m=0,1,2,\dots$$  Then
there exists a piecewise continuous function
$\varphi:\R^2_+\longrightarrow\D$ whose partial derivatives are
complex measures, which satisfies $\varphi\equiv w_n$ on
$\Delta(z_n,1/10)$, for $n=0,1,2,\dots$ and
\begin{itemize}
\item[$a)$] The support of the measure $|\nabla\varphi|$ is
contained in $\cup T(I)\cup\partial_vQ(I)$, where the union is
taken over all intervals $I\in\G$.  Here $\partial_vQ(I)$ means
the vertical part of the boundary of $Q(I)$. \item[$b)$] The
function $\varphi$ has a vertical derivative at any $z\in\R^2_+$
and
$$\dfrac{(\Ima z)|\partial_y\varphi(z)|}{(1-|\varphi(z)|^2)}\leq C\varepsilon.$$
\item[$c)$] The measure
$|\partial_y\varphi(z)|(1-|\varphi(z)|^2)^{-1}dA(z)$ is a Carleson
measure with Carleson norm at most $C\varepsilon$. \item[$d)$]
There exists a universal constant $C_2>0$ such that
$$\dfrac{|\partial_x\varphi(z)|}{1-|\varphi(z)|^2}\leq C_2\varepsilon\sum ds(\partial_vQ(I))(z)$$
as positive measures, where the sum is taken over all $I\in\G$.
Here $ds(\partial_vQ(I))$ means the linear measure in
$\partial_vQ(I)$.
\end{itemize}
\end{lma}
\begin{proof}
Given the interpolation problem for the sequence $\{z_n\}$ we
first extend it to a suitable interpolation problem on a bigger
sequence defined in terms of the intermediate family. We proceed
as in \cite{BN}. To this end, observe that the intermediate family
$\G$ can be viewed as a tree. The root node corresponds to the
unit interval (this is the reason why the point $z_0=1/2+i3/2$ was
added to the sequence). Every interval $I\in\G$ corresponds to a
node $a(I)$ in the tree. Two nodes $a(I)$ and $a(\tilde{I})$ with
$I,\tilde{I}\in\G$ and $\tilde{I}\subset I$ are joined by an edge
in the tree if $\tilde{I}$ is maximal among all $J\in\G$ with
$J\subset I$. As usual, the distance between two nodes in the tree
is defined to be the minimal number of edges joining them. Observe
that the condition $(b2)$ in Lemma B tells that the distance
between two edges $a(I), a(\tilde{I})$ in the tree is bigger than
a fixed multiple of the hyperbolic distance between the associated
points $z(I), z(\tilde{I})$.  We consider the set $\V$ of vertices
of the tree corresponding to points of the original sequence
$\{z_n\}$ and we define a function $\varphi_0$ on $\V$ by
$\varphi_0(a(I))=w_n$ if $z_n$ is the center of $T(I)$. Then the
compatibility condition
$\beta(w_n,w_m)\leq\varepsilon\beta(z_n,z_m)$, and $(b2)$ in Lemma
B gives that the function $\varphi_0:\V\longrightarrow\D$ is
Lipschitz with constant $C_1\varepsilon$ when considering the
metric of the tree in $\V$ and the hyperbolic metric in the image
domain $\D$. We can apply here again the Lemma by McShanne and
Valentine enounced in the previous section to extend the function
$\varphi_0$ to the whole tree $\{a(I):I\in\G\}$.  The extended
function will be also denoted by $\varphi_0$.

\medskip

\noindent Let us now return to the upper half plane.  Let
$\Lambda=\{z_n\}$ be the original sequence and let
$\Lambda^+=\{z(I):I\in\G\}$ be the sequence formed by the center
of the Carleson squares corresponding to intervals in the
intermediate family $\G$, and rename $\Lambda^+=\{z_n^+\}$.  If
$z_n^+=z(I),\ I\in\G$, we denote by
$w_n^+=\varphi_0(z_n^+)=\varphi_0(a(I))$. Notice that if
$z_n^+,z_m^+\in\Lambda^+$ correspond to two consecutive nodes in
the tree, then $\beta(w_n^+,w_m^+)\leq C_2\varepsilon$.  We define
a function $\varphi_1:\R_+^2\longrightarrow\D$ as
$\varphi_1(z_0)=0$ and $\varphi_1(z)=w_n^+$ if $z\in
Q(z_n^+)\backslash\cup Q(z_m^+)$, where the union is taken over
all $Q(z_m^+),\ z_m^+\in\Lambda^+$, contained in $Q(z_n^+)$.  We
next proceed as in \cite[p. 75]{Seip} to smooth the function
$\varphi_1$ in the vertical direction.

\medskip

\noindent Let $Q_0$ be the unit square and set
$$\psi_0(x+iy)=\chi_{Q_0}(x+iy)\min\Big\{1,6(1-y)\Big\}.$$
For each point $z_n^+\in\Lambda^+$ we denote by $Q(n)$ the dyadic
square such that $z_n^+\in T(Q(n))$.  If $Q(n)=Q([a_n,b_n))$, we
define $\tau_n(z)=(z-a_n)/(b_n-a_n)$ and
$\psi_n(z)=\psi_0(\tau_n(z))$. So $\psi_n$ vanishes outside the
square $Q(n)$, has constant value 1 in $Q(n)\cap\{z:\Ima
z<5(b_n-a_n)/6\}$ and it is linear in the vertical direction in
$Q(n)\cap\{z:\Ima z\geq5(b_n-a_n)/6\}$.

\medskip

\noindent We set $b_0^+=w_0=0$ and
$b_n^+=w_n^+-w_m^+=\varphi_1(z_n^+)-\varphi_1(z_m^+),\ n\geq1$,
where $I(z_m^+)$ is the smallest dyadic interval in the family
$\G$ which contains the interval $I(z_n^+)$.  In other words,
$z_m^+$ corresponds to the vertex in the tree above $z_n^+$.  Let
us consider
$$\varphi(z)=\sum_n b_n^+\psi_n(z),\ \ z\in\R^2_+.$$
It is clear that $\partial_x\varphi$ and $\partial_y\varphi$, both
as distributions, are complex measures.  Observe that $\varphi$
has a pointwise vertical partial derivative at any point of the
upper half plane and
$$\partial_y\varphi(z)=0,\ \text{for
}z\notin\bigcup Q(I)\cap\Big\{z:\frac{5}{6}\ell(I)<\Ima z\Big\},$$
while
$$\partial_x\varphi(z)=0,\ \text{for
}z\notin\bigcup
\partial_v Q(I).$$
Here both unions are taken over all intervals $I\in\G$.  So $a)$
holds. Let us now check $b)$.  Observe that if $I(m)\in\G$ and
$I(n)\subseteq I(m)$ is a maximal interval among the ones in $\G$
contained in $I(m)$, then by construction we have
$\beta(\varphi(z_n^+),\varphi(z_m^+))\leq C_2\varepsilon$.  Since
$\varphi$ is linear in the vertical direction, we deduce that
$$\beta(\varphi(z),\varphi(\tilde{z}))\leq C_3\varepsilon\beta(z,\tilde{z}),$$
whenever $z,\tilde{z}\in T(I(n))$ for some $I(n)\in\G$.  Dividing
the equality by $\beta(z,\tilde{z})$ and taking the limit when
$\tilde{z}$ tends to $z$ along the vertical direction we have that
$$\frac{(\Ima z)|\partial_y\varphi(z)|}{(1-|\varphi(z)|^2)}\leq C_4\varepsilon.$$

\noindent The property $c)$ is a direct consequence of $b)$,
property $(b1)$ in Lemma B and the fact that the support of the
measure $\partial_y\varphi$ is contained in $\cup T(I)$, where the
union is taken over all intervals $I\in\G$. Notice that the
behavior of the measure $\partial_x\varphi$ is worse because
$\varphi$ may have jumps across the vertical sides of $Q(I)$,
$I\in\G$. However, since at each step the jump is of a fixed
hyperbolic length, if we consider $z$ and $\tilde{z}$ close
enough, then $\beta(\varphi(z),\varphi(\tilde{z}))$ is at most
$C\varepsilon$ times the number of dyadic intervals $I$ in the
intermediate family $\G$ such that $\partial_v Q(I)$ separates $z$
and $\tilde{z}$. Hence
$$\dfrac{|\partial_x\varphi(z)|}{(1-|\varphi(z)|^2)}\leq C_5\varepsilon\sum ds(\partial_vQ(I))(z)$$
as positive measures, where the sum is taken over all $I\in\G$.
\end{proof}

\subsection{Averaging}\label{sect_average}
The next step of the proof will be smoothing the function
$\varphi$ defined in the above paragraph, so that
$|\partial_x\varphi(z)|/(1-|\varphi(z)|^2)$ verifies also
properties $b)$ and $c)$ in Lemma \ref{nonsmooth_fnct}. As proved
in Section \ref{sect_ext_phi} we may assume that $\Lambda=\{z_n\}$
is a well separated sequence consisting of centers of a certain
collection of dyadic squares. So assume that $\beta(z_n,z_m)>5$,
thereby for each $n=1,2,\dots$ we can add to the sequence
$\Lambda$ the points $z_n^-:=z_n-4\Ima z_n/3$ and
$z_n^+:=z_n+4\Ima z_n/3$. These points are respectively the
centers of the two dyadic squares of the same generation adjacent
to $Q(n)$, denoted by $Q(n)^-$ and $Q(n)^+$. The extended sequence
$\Lambda\cup\{z_n^-\}\cup\{z_n^+\}$, which will be also denoted by
$\Lambda$, will still be a separated sequence. For each
$n=1,2\dots$, attached to the extra points $z_n^-$ and $z_n^+$
consider the corresponding value $w_n$ in the sequence $\{w_n\}$.
Hence the function $\varphi$ constructed in the previous section
verifies that $\varphi(z_n)=\varphi(z_n^-)=\varphi(z_n^+)=w_n$,
for n=1,2\dots  This is a trick used by C.\ Sundberg in
\cite{Sun}.

\medskip

\noindent For $0\leq t<1$ let now $\Lambda_t$ be the sequence
obtained by translating the sequence $\Lambda$ by $t$ euclidean
units, that is, $\Lambda_t=\big\{z_n+t : n=0,1,2,\ldots\big\}$.
Let $\G_t$ be the intermediate family of dyadic intervals given in
Lemma B associated to the sequence $\Lambda_t$. Finally let
$\varphi_t$ be the function given by Lemma \ref{nonsmooth_fnct}.
Notice that we could also have looked at it from an equivalent
point of view: fix the sequence $\Lambda$ and let $\Di_t=\big\{I_t
: I_t-t\in\Di\big\}$ be the translation of the standard dyadic
intervals by $t$ euclidean units. We could have considered the
analogues of Lemma \ref{nonsmooth_fnct} for the sequence $\Lambda$
and the translated dyadic family $\Di_t$.

\medskip

\noindent Observe that for any $z\in\R^2_+$ the set of possible
values $\big\{\varphi_t(z+t) : 0\leq t<1\big\}$ is bounded.  Then
we may apply results in Section \ref{sect_center} to define
$\varphi(z)$. Take as measure space the unit interval $[0,1)$ with
the Lebesgue measure, and as Hadamard space the hyperbolic disc
$\D$. For each $z\in\R^2_+$ consider the mapping
$\varphi(z,\cdot):[0,1)\longrightarrow\D$ defined as
$\varphi(z,t)=\varphi_t(z+t)$, $t\in[0,1)$. Then we define
$\varphi(z)$ as the center of mass of the pushforward measure
$\mu$ on $\D$ defined as $\mu(E)=|\{t\in[0,1):\varphi_t(z+t)\in
E\}|$, $E\in\R^2_+$. Now we will prove that the function $\varphi$
verifies properties $(A)$, $(B)$ and $(C)$.

\bigskip

\noindent In Figure \ref{figure2}, the left hand side represents a
point $z_n$ of the original sequence and its two corresponding
extra points $z_n^-$ and $z_n^+$. The right hand side corresponds
to translate the sequence by $t$ euclidean units. Since we have
added the extra points to the original sequence, for every
$t\in[0,1)$ the function $\varphi_t$ verifies that
$\varphi_t(z+t)=w_n$ for every $z\in Q(n)\cup Q(n)^-\cup Q(n)^+$
with $\ell(Q)/2<\Ima z<5\ell(Q)/6$. In other words, for every
$t\in[0,1)$ the function $\varphi_t(z+t)$ has constant value $w_n$
in the grey strip of Figure \ref{figure2}. In particular,
$\varphi_t(z+t)=w_n$ for every $z\in\Delta(z_n,1/10)$ and every
$t\in[0,1)$. This proves $(A)$.
\begin{figure}[htbp]

  \psfrag{Q1}[c][c]{$Q(n)^-$}
  \psfrag{Q2}[c][c]{$Q(n)$}
  \psfrag{Q3}[c][c]{$Q(n)^+$}
  \psfrag{Q4}[c][c]{$Q_t(n)^-$}
  \psfrag{Q5}[c][c]{$Q_t(n)$}
  \psfrag{Q6}[c][c]{$Q_t(n)^+$}
  \psfrag{Z1}[c][c]{$z_n^-$}
  \psfrag{Z2}[c][c]{$z_n$}
  \psfrag{Z3}[c][c]{$z_n^+$}
  \psfrag{Z4}[c][c]{$z_n^-+t$}
  \psfrag{Z5}[c][c]{$z_n+t$}
  \psfrag{Z6}[c][c]{$z_n^++t$}

  \begin{center}

  \includegraphics[width=1\textwidth]{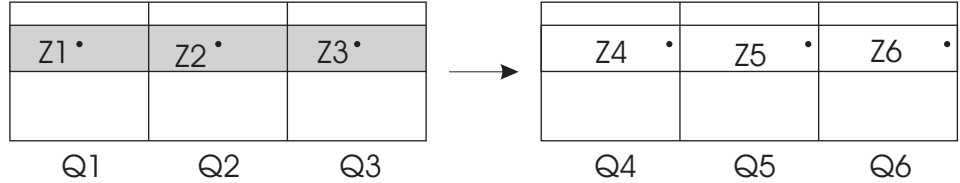}

\end{center}
\caption{The function $\varphi_t(z+t)$ has constant value $w_n$ in
the grey strip for every $t\in[0,1)$.}
   \label{figure2}
\end{figure}

\noindent Now we can check properties $(B)$ and $(C)$. Applying
Proposition \ref{convex_ineq} to the mappings $\varphi(z,\cdot)$
and $\varphi(w,\cdot)$ we obtain that
\begin{equation}\label{integral_ineq}
\beta(\varphi(z),\varphi(w))\leq\int_0^1\beta(\varphi_t(z+t),\varphi_t(w+t))\
dt.
\end{equation}
If $|\nabla\varphi(z)|$ was replaced by $|\partial_y\varphi(z)|$,
then properties $(B)$ and $(C)$ would follow from estimate
(\ref{integral_ineq}) and the corresponding properties in Lemma
\ref{nonsmooth_fnct}. Property $d)$ in Lemma \ref{nonsmooth_fnct}
and the inequality (\ref{integral_ineq}) will provide a similar
estimate for $|\partial_x\varphi(z)|$.

\medskip

\noindent First let $z,w\in\R_+^2$ with $\Ima z=\Ima w$, and take
$k\in\Z$ so that $2^{-k-1}< \Ima z\leq 2^{-k}$.  It is enough to
prove $(B)$ for $z,w\in\R_+^2$ such that $|z-w|\leq \Ima z$.  For
$j=1,\dots,k$ let $A_j$ be the set of $t\in[0,1)$ such that there
exists $I\in\G_t$, $|I|=2^{-j}$ such that at least one of the two
vertical sides of $Q(I)$ is between $z+t$ and $w+t$. Since
$\varphi_t$ can jump at most $C\varepsilon$ hyperbolic units in
$T(I)$ with $I\in\G_t$, then if $t\in A_j$ we have that
$$\beta(\varphi_t(z+t),\varphi_t(w+t))\leq C_2\varepsilon(k-j).$$
Since $|A_j|\leq 2^{j}|z-w|\leq C 2^{j-k}\beta(z,w)$, we deduce
that
$$\beta(\varphi(z),\varphi(w))\leq C_3\varepsilon\sum_{j=1}^k(k-j)2^{j-k}\beta(z,w)\leq C_4\varepsilon\beta(z,w).$$
So $(B)$ holds, and hence $\varphi$ is differentiable almost
everywhere in $\D$. To check now property $(C)$ for
$|\partial_x\varphi(z)|$, let $Q$ be a Carleson square. Divide by
$\beta(z,w)$ in both sides of estimate (\ref{integral_ineq}).
Making $w$ tend to $z$ and then using property $d)$ in Lemma
\ref{nonsmooth_fnct} we obtain that
\begin{align*}
\int_Q\frac{|\partial_x\varphi(z)|}{1-|\varphi(z)|^2}&\leq\int_0^1\left(\int_Q\frac{|\partial_x\varphi_t(z+t)|}{1-|\varphi_t(z+t)|^2}\right)dt\\
&\leq C_2\varepsilon\int_0^1\sum \ell(\partial_vQ(I)\cap(Q-t))dt,
\end{align*}
where the sum is taken all over the dyadic interval $I\in\G_t$.
Split the above sum into two terms, the last integral can be
written as $(P_1)+(P_2)$, where $(P_1)$ corresponds to those
intervals $I\in\G_t$ such that $|I|<\ell(Q)$ and $(P_2)$ to those
such that $|I|\geq\ell(Q)$.  Applying property $(b1)$ in Lemma B
to $(P_1)$, for any $t\in[0,1)$ we have
$$\sum\ell(\partial_vQ(I)\cap(Q-t))\leq C\ell(Q),$$
where the sum is taken over all $I\in\G_t$ with $|I|<\ell(Q)$.
Hence $(P_1)\leq C\ell(Q)$.  To estimate $(P_2)$ we proceed as in
property $(B)$. Let $k\in\Z$ so that $2^{-k-1}<\ell(Q)\leq2^{-k}$.
For each dyadic interval $I$ of generation $j=1,\dots,k$ let
$A(I)$ be the set of $t\in[0,1)$ such that $I\in\G_t$ and
$\partial_v Q(I)\cap(Q-t)\neq\emptyset$. Take $A_j=\cup A(I)$,
where the union is taken over all dyadic intervals of generation
$j$. Notice that $|A_j|\leq2^{j-k}$. Hence
$$(P_2)\leq\sum\ell(Q)|A(I)|,$$
where the sum is taken over the dyadic intervals $I$ of generation
smaller than $k$.  Hence
$$(P_2)\leq\ell(Q)\sum_{j=1}^k|A_j|\leq2\ell(Q).$$
So property $(C)$ holds.

\subsection{A suitable $\bar\partial$ problem}\label{sect_deltabarr}
The main purpose of this section is to construct an analytic
interpolating function.  The main tools will be a certain
$\bar\partial$-equation and $BMO$ techniques.  We start with the
following well known result (see \cite{Var}) whose proof is
presented for the sake of completeness.  Recall that $\Bm(\R)$ is
defined as the set of functions $f\in \Lu(\R)$ such that
$$\|f\|_{\Bm}=\sup\frac{1}{|I|}\int_I|f(x)-f_I|dx,$$
where $f_I=\dfrac{1}{|I|}\displaystyle\int_I f$ and the supremum
is taken over all intervals $I\subset\R$. Let us also denote by
$P_z(f)$ the Poisson integral of $f$ at the point $z$.

\begin{lma}\label{BMO_result}
Let $F$ be a smooth function in $\R^2_+$ such that
$$f(x)=\lim_{y\rightarrow0}F(x+iy)$$
exists almost every $x\in\R$.  Assume $f\in \Lu(\R)$ and $|\nabla
F(z)|dA(z)$ is a Carleson measure.  Then $f\in \Bm(\R)$ and
$$\|f\|_{\Bm}\leq C \| |\nabla F(z)|dA(z) \|_C.$$
Moreover, for any $z\in\R^2_+$ there exists a point $\tilde{z}\in
T(I(z))$ such that $$|P_z(f)-F(\tilde{z})|\leq C \||\nabla
F(z)|dA(z)\|_{C}.$$
\end{lma}
\begin{proof}
Let $I\subset\R$ be an interval and let $z_I$ be the center of the
square $T(I)$, and $x_I$ the center of the interval $I$. Since
$|\nabla F(z)|dA(z)$ is a Carleson measure we have that
$$\int_{T(I)}|\nabla F(z)|dA(z)=\int_{|I|/2}^{|I|}\int_I |\nabla F(x+it)|\,dx\,dt\leq \|\nabla F\|_{C}|I|.$$
Hence there exists $t_0\in\left(|I|/2,|I|\right)$ such that
$\displaystyle\int_I|\nabla F(x+it_0)|\,dx\leq 2\|\nabla F\|_{C}$,
so for any $x\in I$ we have that
$$|F(x+it_0)-F(x_I+it_0)|\leq\int_I |\nabla F(s+it_0)|\,ds\leq 2\|\nabla F\|_{C}.$$
Moreover for any $x\in I$ we have that
$$|f(x)-F(x+it_0)|\leq\int_{0}^{t_0} |\nabla F(x+is)|\,ds.$$
The last two inequalities imply that
\begin{multline*}
\frac{1}{|I|}\int_I|f(x)-F(x_I+it_0)|\,dx\leq\\\leq\frac{1}{|I|}\int_I|f(x)-F(x+it_0)|\,dx+\frac{1}{|I|}\int_I|F(x+it_0)-F(x_I+it_0)|\,dx\leq\\
\leq\frac{1}{|I|}\int_{Q(I)}|\nabla F(z)|\,dA(z)\,+\,2\|\nabla
F\|_{C}\leq 3\|\nabla F\|_{C},
\end{multline*}
which shows that $f\in \Bm$ and that $\|f\|_{\Bm}\leq C \|\nabla
F\|_C$. The last estimate in the theorem follows from the well
known fact that
$$\left|\frac{1}{|I|}\int_I f - P_{z_I}(f)\right|\leq C\|f\|_{\Bm}.$$
\end{proof}

\medskip

\noindent Let $\{z_n\}$ be a separated sequence satisfying the
density condition $(b)$ in Theorem \ref{main_theorem}. Given a
sequence of values $\{w_n\}$ satisfying the compatibility
condition (\ref{compatibility_condition}), we next construct a
function $f\in\Hi$ with $\|f\|_\infty\leq 1$ and $f(z_n)=w_n$ for
$n=1,2,\dots$ Let $\varphi$ be the smooth interpolating function
constructed in subsections \ref{sect_non-smooth_funct} and
\ref{sect_average}. We will apply Lemma \ref{BMO_result} to the
function $F(z)=\log(1-|\varphi(z)|^2)$, which by $(C)$ satisfies
$$\| |\nabla F(z)|dA(z) \|_C\leq C\varepsilon.$$
Actually, since the function $F$ verifies that
$|F(z)-F(\tilde{z})|\leq C$ if $\beta(z,\tilde{z})\leq 1$, the
last estimate in Lemma \ref{BMO_result} gives that there exists a
universal constant $C_1>0$ such that
\begin{equation}\label{poisson_estimate}
|P_z(\log(1-|\varphi|^2))-\log(1-|\varphi(z)|^2)|\leq
C_1\varepsilon,\ \ z\in\R^2_+.
\end{equation}

\bigskip

\noindent Let $E(1-|\varphi|^2)$ be the outer function given by
\begin{equation}\label{outer_function}
E(1-|\varphi|^2)(z)=\operatorname{exp}\Big(\int_\R\frac{-i}{x-z}\log(1-|\varphi(x)|^2)\
dx\Big),\ \ z\in\R_+^2.
\end{equation}
The estimate (\ref{poisson_estimate}) reads
\begin{equation}\label{outer_estimate}
e^{-C_1\varepsilon}\leq\frac{|E(1-|\varphi|^2)(z)|}{1-|\varphi(z)|^2}\leq
e^{C_1\varepsilon},\ \ z\in\R_+^2.
\end{equation}

\noindent Therefore, again by $(C)$ we have that
$$\frac{|\nabla\varphi(z)|}{E(1-|\varphi|^2)(z)}dA(z)$$
is a Carleson measure with Carleson norm bounded by
$C_2\varepsilon$.  Assuming that $\{z_n\}$ is a separated
sequence, the rest of the proof is fairly standard. Let $B(z)$ be
the Blaschke product with zeros $\Lambda=\{z_n\}$. Since $\{z_n\}$
is an interpolating sequence for $\Hi$, there exists $C_3>0$ such
that $|B(z)|\geq C_3$ for $z\notin\Delta(z_n,C)$. Recall that
$\varphi$ is constant on the hyperbolic discs $\Delta(z_n,C)$,
hence the measure
$$\left|\frac{\bar{\partial}\varphi(z)}{B(z)E(1-|\varphi|^2)(z)}\right|dA(z)$$
is a Carleson measure with Carleson norm smaller than
$C_4\varepsilon$.  Hence we can find a smooth function $b$ in
$\R_+^2$ which extends continuously to $\R$ with
$\|b\|_{\Li(\R)}\leq C\varepsilon$, such that
$$\bar{\partial}b(z)=\frac{2\bar{\partial}\varphi(z)}{B(z)E(1-|\varphi|^2)(z)},\ \
z\in\R^2_+.$$ See \cite[p. 311]{Garnett}. Then
$f=\varphi-2^{-1}BE(1-|\varphi|^2)b$ is an analytic function and
at almost every $x\in\R$ we have
\begin{equation*}
|f(x)|\leq|\varphi(x)|+2^{-1}(1-|\varphi(x)|^2)C\varepsilon\leq1
\end{equation*}
if $\varepsilon$ is taken sufficiently small, so that
$C\varepsilon\leq 1$. So, under the assumption that $\{z_n\}$ is a
separated sequence, we have constructed a function $f\in\B$
fulfilling the interpolation.

\bigskip

\noindent For later purposes it will be useful to state the
following fact.

\begin{obs}\label{obs_analytic}
There exists a function $f\in\B$ satisfying $f(z_n)=w_n$, for
$n=1,2\dots$ and
\begin{equation}\label{outer_analytic}
\left|E(1-|f|)(z)\right|\geq C(1-|f(z)|),\, z\in\R^2_+.
\end{equation}
Moreover, there exists a constant $\eta>0$ depending on $\{z_n\}$
such that
\begin{equation}\label{lip_analytic}
\beta(f(z),f(z_n))\leq C\varepsilon\beta(z,z_n),\ \ z\in
\Delta(z_n,\eta)
\end{equation}
for any $n=1,2,\dots$
\end{obs}
\noindent To show (\ref{outer_analytic}) notice that since
$\|b\|_\infty\leq C\varepsilon<1$ there exists a constant $C_1>0$
such that $1-|f(x)|\geq C_1(1-|\varphi(x)|)$ for any $x\in\R$.
Hence $|E(1-|f|)(z)|\geq C_2|E(1-|\varphi|)(z)|$ for any
$z\in\R^2_+$. Applying (\ref{outer_estimate}) we deduce that
$$|E(1-|f|)(z)|\geq C_3(1-|\varphi(z)|)$$
for any $z\in\R_+^2$.  Since $|\varphi(z)-f(z)|\leq
C\varepsilon(1-|\varphi(z)|)$, then (\ref{outer_analytic}) holds.

\medskip

\noindent The proof of (\ref{lip_analytic}) is more subtle and
depends on a beautiful result by P. Jones on bounded solutions of
$\bar\partial$-equations.

\begin{lma}\label{p_jones_sol}
Let $F$ be a continuous function in the upper half plane such that
$|F(z)|dA(z)$ is a Carleson measure.  Let $\{z_n\}$ be a sequence
of points in the half plane $\R^2_+$.  Assume that there exists
$\delta>0$ such that $F\equiv 0$ on $\Delta(z_n,\delta)$.  Then
there exists a function $b\in\mathcal{C}(\R^2_+)$ with
$\bar\partial b=F$ and
\begin{equation}\label{Jones_estimate}
\|b\|_{\Li(\R)}+\sup\{|b(z)|:z\in\cup \Delta(z_n,\delta/2)\}\leq
C(\delta)\||F(z)|dA(z)\|_C.
\end{equation}

\end{lma}
\begin{proof}
Without loss of generality we may assume that
$\||F(z)|dA(z)\|_C=1$.  P. Jones \cite{Jones} found an explicit
solution of the $\bar\partial$-equation with uniform estimates.
This formula is
$$b(z)=\frac{1}{\pi}\int_{\R_+^2}\frac{(\Ima\xi)F(\xi)}{(\xi-z)(z-\bar\xi)}K(\xi,z)dA(\xi),$$
where
$$K(\xi,z)=\exp\left\{\int_{S(\xi)}\left(\frac{i}{\xi-\bar w}-\frac{i}{z-\bar w}\right)|F(w)|dA(w)\right\},$$
and $S(\xi)=\{w\in\R^2_+:\Ima w<\Ima \xi\}$. The estimate
$\|b\|_{\Li(\R)}\leq C$ is proved in \cite{Jones}. We will use the
same argument that we are using to estimate the second term in the
left hand side of (\ref{Jones_estimate}). Write $E=\D\setminus\cup
\Delta(z_n,\delta)$ and notice that $|\xi-z|\geq
C_1(\delta)|\bar\xi-z|$ if $\xi\in E$ and $z\in\cup
\Delta(z_n,2^{-1}\delta)$.  Hence for such points $z\in\cup
\Delta(z_n,2^{-1}\delta)$ we have
$$|b(z)|\leq C_2(\delta)\int_E\frac{(\Ima\xi)|F(\xi)|}{|\bar\xi-z|^2}|K(\xi,z)|dA(\xi).$$
Since any $w\in S(\xi)$ verifies $\Ima w> \Ima\xi$, we deduce that
$$\int_{S(\xi)}\frac{\Ima
w}{|\bar w-\xi|^2}|F(w)|dA(w)\leq
\int_{\R^2_+}\frac{\Ima\xi}{|\bar w-\xi|^2}|F(w)|dA(w)\leq C$$ for
any $\xi\in\R^2_+$. Then
$$|b(z)|\leq C_3\int_{\R^2_+}\frac{(\Ima\xi)|F(\xi)|}{|\bar\xi-z|^2}\exp\left(\int_{S(\xi)}\frac{-(\Ima w)|F(w)|}{|\bar w-z|^2}dA(w)\right)dA(\xi),$$
where $C_3=C_3(\delta)$.  Arguing as in \cite{Jones}, the integral
above compares to $\displaystyle\int_0^\infty e^{-x}dx=1$, and the
proof is completed.
\end{proof}

\noindent We now continue with the proof of (\ref{lip_analytic}).
Recall that $\varphi$ was constant on hyperbolic discs centered at
the points $\{z_n\}$ of a fixed radius.  Hence the function
$F=\bar\partial\varphi/(BE(1-|\varphi|^2))$ satisfies the
conditions of Lemma \ref{p_jones_sol}, so let $b$ be the solution
given by this Lemma. Pick $\eta\leq2^{-1}\delta$ and take $z\in
\Delta(z_n,\eta)$ for some fixed $n$.  Then $|b(z)|\leq
C_4\varepsilon$ and we deduce
$$|f(z)-\varphi(z)|\leq C_4\varepsilon|B(z)||E(1-|\varphi|)(z)|\leq C_5\varepsilon|B(z)|(1-|\varphi(z)|).$$
Since
$$|B(z)|\leq \left|\frac{z-z_n}{z-\bar z_n}\right|\leq C_6\beta(z_n,z),$$
we have that
$$\beta(f(z),\varphi(z))\leq C_7\varepsilon\beta(z_n,z).$$
Since $\beta(\varphi(z),\varphi(w))\leq C\varepsilon\beta(z,w)$
for any $z,w,\in\R^2_+$, the estimate (\ref{lip_analytic})
follows.

\subsection{Union of two separated sequences.}\label{sect_suff_2seq}
This section is devoted to proving the sufficiency of conditions
$(a)$ and $(b)$ in Theorem \ref{main_theorem}. This will end the
proof of Theorem \ref{main_theorem}.  So let
$\{z_n^{(1)}\}\cup\{z_n^{(2)}\}$ be the union of two separated
sequences verifying the density condition $(b)$. Pick a number
$\delta>0$ smaller than the quantifier $\eta$ appearing in
(\ref{lip_analytic}) as well as smaller than the separation
constant of the sequence $\{z_n^{(1)}\}$.  Adding the points in
$\{z_n^{(2)}\}\setminus\cup \Delta(z_n^{(1)},\delta)$ to the
sequence $\{z_n^{(1)}\}$, we can assume that the sequence
$\{z_n^{(2)}\}$ is contained in $\cup \Delta(z_n^{(1)},\delta)$.
So for each $z_k^{(2)}\in\{z_n^{(2)}\}$ there is a point in
$\{z_n^{(1)}\}$, denoted by $z_{n(k)}^{(1)}$, with
$\beta(z_{n(k)}^{(1)},z_k^{(2)})\leq\delta$.  Let $B$ be the
Blaschke product with zeros $\{z_n^{(1)}\}$.  Since
$\{z_n^{(1)}\}$ is a separated sequence which satisfies the
density condition $(b)$ in Theorem \ref{main_theorem}, it is an
interpolating sequence for $\Hi$. Hence there exists a constant
$C_1>0$ such that
$$B(z_k^{(2)})\geq C_1\rho(z_{n(k)}^{(1)},z_k^{(2)}),\ \ k=1,2,\dots,$$
where $\rho(a,b)=|a-b|/|a-\bar b|$ is the pseudohyperbolic
distance in $\R_+^2$.  Now let $\{w_n^{(1)}\}\cup\{w_n^{(2)}\}$ be
a sequence of values in the unit disc that satisfy the
compatibility condition (\ref{compatibility_condition}).  Then
there exists $f\in\B$ with $f(z^{(1)}_n)=w^{(1)}_n$, for
$n=1,2\dots$, and satisfying the conditions given in Remark
\ref{obs_analytic} for the sequence $\{z_n^{(1)}\}$. However,
since $\beta(z_{n(k)}^{(1)},z_k^{(2)})\leq\delta$ for any
$k=1,2,\dots$, we can assume that conditions
(\ref{outer_analytic}) and (\ref{lip_analytic}) hold for the whole
sequence $\{z_n^{(1)}\}\cup\{z_n^{(2)}\}$, once the constant $C$
is replaced by other absolute constants.  Notice that estimate
(\ref{outer_analytic}) gives
$$\left|\frac{w_n^{(2)}-f(z_n^{(1)})}{B(z_n^{(2)})E(1-|f|)(z_n^{(2)})}\right|\leq C_2\frac{\rho(w_n^{(2)},w_n^{(1)})}{\rho(z_n^{(2)},z_n^{(1)})},\ \ n=1,2,\dots$$
Since $\beta(a,b)$ is comparable to $\rho(a,b)$ whenever
$\beta(a,b)\leq1$, the compatibility condition
(\ref{compatibility_condition}) yields
$$\sup_n\left|\frac{w_n^{(2)}-f(z_n^{(1)})}{B(z_n^{(2)})E(1-|f|)(z_n^{(2)})}\right|\leq C_3\varepsilon.$$
Also, using (\ref{lip_analytic}) instead of the compatibility
condition, the argument above tells that
$$\sup_n\left|\frac{f(z_n^{(2)})-f(z_n^{(1)})}{B(z_n^{(2)})E(1-|f|)(z_n^{(2)})}\right|\leq C_4\varepsilon.$$
So the values
$$w_n^*=\frac{w_n^{(2)}-f(z_n^{(2)})}{B(z_n^{(2)})E(1-|f|)(z_n^{(2)})},\ \ n=1,2,\dots$$
satisfy $\sup_n|w_n^*|\leq 2C_5\varepsilon$.  Since
$\{z_n^{(2)}\}$ is a separated sequence satisfying the density
condition $(b)$ in Theorem \ref{main_theorem}, it is an
interpolating sequence for $\Hi$. Hence, fixing $\varepsilon>0$
sufficiently small there exists $g\in\B$ such that
$g(z_n^{(2)})=w_n^*$, for $n=1,2,\dots$ Then the function
$$h=f+BgE(1-|f|)$$
is in $\B$ and will interpolate the whole sequence of values
$\{w_n^{(1)}\}\cup\{w_n^{(2)}\}$ at the whole sequence
$\{z_n^{(1)}\}\cup\{z_n^{(2)}\}$. This ends the proof of Theorem
\ref{main_theorem}.

\subsection{Remark.} It is worth mentioning that if $\{z_n\}$ is a sequence
that verifies $(a)$ and $(b)$ in Theorem \ref{main_theorem} and
$\{w_n\}$ is a sequence of values in the unit disc which satisfy
the compatibility condition
$$\beta(w_n,w_m)\leq\varepsilon\beta(z_n,z_m),\ \ n,m=1,2,\dots$$
our construction provides a non extremal point $f$ in $\B$ with
$f(z_n)=w_n$ for $n=1,2,\dots$  Then applying a refinement of
Nevanlinna's theorem due to A.~Stray \cite{Stray} we may find a
Blaschke product $I$ such that $I(z_n)=w_n$, for $n=1,2,\dots$

\bibliographystyle{plain}
\bibliography{paper}

\end{document}